\documentclass[11pt,reqno]{amsart} 
\usepackage{amsmath,amssymb,rawfonts}
\usepackage{tikz}
\usepackage{pb-diagram}
\newcommand{\impl}{\longmapsto} 
\newcommand{\ls}{\leqslant}
\newcommand{\be}{\begin{equation}}
\newcommand{\ee}{\end{equation}}
\newcommand{\ix}{i_{\text{\tiny{$X$}}}}

\newtheorem{theorem}{Theorem}[section]

\newtheorem{defi}[theorem]{Definition}
\newtheorem{prop}[theorem]{Proposition}
\newtheorem{coro}[theorem]{Corollary}

\newtheorem{ex}[theorem]{Example}

\begin{document} 
\title{VARIABLE-BASIS FUZZY INTERIOR OPERATORS}
\author{Joaqu\'in Luna-Torres $^{\lowercase{a}}$\,\ and \,\ Lilibeth De Horta Narvaez $^{\lowercase{b}}$}

\dedicatory{Escuela de Matem\'aticas, Universidad Sergio Arboleda,\\ Calle 74 No. 14-14, Bogot\'a, Colombia}
\email{$^b$ldehortana@yahoo.es}
\email{$^a$ jluna@ima.usergioarboleda.edu.co}

\begin{abstract}
For a topological space it is well-known that the associated closure and interior operators provide equivalent descriptions of set-theoretic  topology; but it is not generally true in other categories, consequently it makes sense to define and study the notion of interior operators $I$ in the context of fuzzy set theory, where we can find categories in a lattice-theoretical context. 
Fuzzy interior operators have been studied by U. H\"ohle, A. \v{S}ostak  and others, (1999), these works were used to describe $L$-topologies on a set $X$.
More recently, M. Diker, S. Dost and A. U$\check{g}$ur (2009 ) present  interior and closure operators on texture spaces in the sense of $\check{C}$ech, and  F. G. Shi(2009) studies interior operators via L-fuzzy neighborhood systems.\\
The aim of this paper is to propose a more general theory of variable-basis fuzzy  interior operators, employing both categorical tools and the lattice theoretical foundations investigated by S. E. Rodabaugh (1999), where the lattices are usually non-complemented.  furthermore, we construct some topological categories.

\end{abstract}
\maketitle 
\baselineskip=1.7\baselineskip
\section*{0. Introduction}
For a topological space it is well-known that the associated closure and interior operators provide equivalent descriptions of set-theoretic  topology; but it is not generally true in other categories, consequently it makes sense to define and study the notion of interior operators $I$ in the context of fuzzy set theory, where we can find categories in a lattice-theoretical context. 
Interior operators are very useful tools in several areas of classical mathematics, its applications such as {\it Geographic information systems} and in general  category theory. In fuzzy set  theory,  fuzzy interior operators have been studied by U. H\"ohle, A. \v{S}ostak  and others, (see e.g. \cite{HS}), these works were used to describe $L$-topologies on a set $X$.

More recently, M. Diker, S. Dost and A. U$\check{g}$ur present  interior and closure operators on texture spaces in the sense of $\check{C}$ech (see \cite{DDU}), and  F. G. Shi( \cite{FS}) studies interior operators via L-fuzzy neighborhood systems. On the other hand, W. Shi and K. Liu (\cite{SL}) present a development of computational fuzzy topology, which is based on fuzzy  interior  and closure operators in order to get topological relations between spatial objects and  Geographic information systems.

The aim of this paper is to propose a more general theory of variable-basis fuzzy  interior operators, employing both categorical tools and the lattice theoretical fundations investigated in  \cite{SER} and \cite{HS}, where tha lattices are usually non-complemented.

The paper is organized as follows: Following \cite{SER} and \cite{HS}  we introduce, in section $1$, the basic lattice theoretical fundations. In section $2$, we present the concept of variable-basis fuzzy interior operators and then we construct a topological category $(\text{VBIO-SET},U)$. 
In section 3, we study some additional  properties of interior operators: idempotent and productive interior operators as well as open fuzzy sets and open morphisms. Finally in section 4, we present some examples of various classes of interior operators. 

\section{From Lattice Theoretic Foundations}
Let $(L, \leq)$ be a complete, infinitely distributive lattice, i.e. $(L, \leq)$ is a partially ordered set such that for every subset $A\subset L$ the join $\bigvee A$ and the meet $\bigwedge A$ are defined, moreover $(\bigvee A) \wedge \alpha = \bigvee \{ a\wedge \alpha) \mid a \in A \}$ and\linebreak $(\bigwedge A) \vee \alpha = \bigwedge \{a\vee \alpha) \mid a \in A \}$
for every $\alpha \in L$.  In particular, $\bigvee L = \top$ and $\bigwedge L = \bot$ are respectively the universal upper and the universal lower bounds in $L$.
We assume that $\bot \ne \top$, i.e.\ $L$ has at least two elements.

\subsection{Complete quasi-monoidal lattices} 
The definition of complete quasi-monoidal lattices introduced by  S. E. Rodabaugh in \cite{SER} is the following:

A $cqm-$lattice (short for complete quasi-monoidal lattice) is a
triple\linebreak $(L,\leqslant,\otimes)$ provided with the
following properties
\begin{enumerate}
\item[(1)] $(L,\leqslant)$ is a complete lattice with upper bound $\top$ and lower bound $\bot$.
\item[(2)] $\otimes: L\times L  \rightarrow L$ is a binary operation satisfying the following axioms:
\begin{enumerate}
\item $\otimes$ is isotone in both arguments, i.e. $\alpha_1 \leqslant\alpha_2,\,\ \beta_1\leqslant \beta_2$ implies
$\alpha_1\otimes\beta_1\leqslant \alpha_2\otimes\beta_2$; \item  $\top$ is idempotent, i.e. $\top \otimes\top = \top$.
\end{enumerate}
\end{enumerate}
The category $CQML$ comprises the following data:
\begin{enumerate}
\item[(a)] {\bf Objects}: Complete quasi-monoidal lattices.
\item[(b)] {\bf Morphisms}: All $SET$ morphisms, between the above objects, which preserve $\otimes$ and $\top$ and arbitrary $\bigvee$.
\item[(c)] Composition and identities are taken from $SET$. 
\end{enumerate}
The category $LOQML$ is the dual of $CQML$, i.e. $LOQML=CQML^{op}$.
\subsection{$GL-$monoids}
A $GL-$monoid (see \cite{HS}) is a complete
lattice enriched with a further binary operation $\otimes$, i.e.\ a triple $(\L, \leq, \otimes)$ such that:
\begin{enumerate}
\item[(1)]
$\otimes$ is isotone, commutative and associative; 
\item[(2)]
$(\L,\leq,\otimes)$ is integral, i.e.\ $\top$ acts as the unity: $\alpha \otimes \top = \alpha$, $\forall \alpha \in \L$;
\item[(3)]
$\bot$ acts as the zero element in $(\L, \leq, \otimes)$, i.e.\ $\alpha\otimes \bot = \bot$, $\forall \alpha \in \L$;
\item[(4)]
$\otimes$ is distributive over arbitrary joins, i.e.\ $\alpha \otimes (\bigvee_{\lambda} \beta_{\lambda}) = \bigvee_{\lambda} (\alpha \otimes \beta_{\lambda})$,
$\forall \alpha \in \L, \forall \{ \beta_{\lambda} : \lambda \in I\} \subset \L$;
\item[(5)]
$(\L, \leq, \otimes)$ is divisible, i.e.\ $\alpha \leq \beta$ implies the existence of $\gamma \in \L$ such that $\alpha = \beta \otimes \gamma$.
\end{enumerate}
It is well known that every $GL-$monoid is residuated, i.e.\ there exists a further binary operation ``$\impl$'' (implication) on $\L$ satisfying the following condition:
$$\alpha \otimes \beta \leq \gamma \Longleftrightarrow \alpha \leq (\beta \impl \gamma) \qquad \forall \alpha, \beta, \gamma \in \L.$$
Explicitly implication is given by
\[
\alpha \impl \beta = \bigvee \{ \lambda \in \L \mid \alpha \otimes \lambda \leq \beta \}.
\]
If $X$ is a set and $L$ is a $GL$-monoid (or a complete quasi-monoidal lattice), then the fuzzy powerset $L^X$ in an obvious way can be pointwise endowed with a structure of a $GL$-monoid (or of a complete quasi-monoidal lattice). In particular the $L$-sets $1_X$ and $0_X$ defined by $1_X (x)= \top$ and $0_X (x) = \bot$ $\forall x \in X$ are respectively the universal  upper and lower bounds in $L^X$.

\subsection{Powerset operator foundations}

We give the powerset operators, developed and justified in detail by S.E. Rodabaugh in \cite{SER} and \cite{SER1}.\
Let $f\in SET(X,Y)$,\,\ $L,M\in |CQML|$, \,\ $\phi\in LOQML(L,M)$, and $\wp(X)$,\ $\wp(Y),\ L^X,\ M^Y$ be the classical powerset of $X$, the classical powerset of $Y$, the $L$-powerset of $X$, and the $M$-powerset of $Y$, respectively. Then the following powerset operators are defined:
\begin{enumerate}
\item
$
f^{\rightarrow}:\wp(X)\rightarrow \wp(Y)\,\ \text{by}\,\ f^{\rightarrow}(A)=\{f(x)\mid x\in A\}
$
\item
$
f^{\leftarrow}:\wp(Y)\rightarrow \wp(X)\,\ \text{by}\,\ f^{\leftarrow}(B)=\{x\in X\mid f(x)\in B\}
$
\item
$
f_L^{\rightarrow}:L^X\rightarrow L^Y\,\ \text{by}\,\ f_L^{\rightarrow}(a)(y)=\bigvee_{f(x)=y} a(x)
$
\item
$
f_L^{\leftarrow}:L^Y\rightarrow L^X\,\ \text{by}\,\ f_L^{\leftarrow}(b)=b\circ f
$
\item
$
^{*}\phi:L\rightarrow M \,\ \text{by}\,\ ^{*}\phi(\alpha)=\bigwedge \{\beta\in M\mid \alpha\leqslant \phi^{op}(\beta)\}
$
\item
$
\langle^{*}\phi\rangle:L^X\rightarrow M^X \,\ \text{by}\,\ \langle^{*}\phi\rangle(a)={^{*}\phi}\circ a
$
\item
$
\langle\phi^{op}\rangle: M^X\rightarrow L^X\,\ \text{by}\,\ {\langle\phi^{op}\rangle}(b)=\phi^{op}\circ b
$
\item
$
\left( f,\Phi \right)^{\rightarrow}:L^X\rightarrow M^Y \,\ \text{by}\,\ \left( f,\Phi \right)^{\rightarrow}(a) = \bigwedge \{b\mid f_L^{\rightarrow}(a)\leqslant \left( \Phi^{op} \right)(b)\},
$
\item
$
\left( f,\Phi \right)^{\leftarrow} : M^Y\rightarrow L^X \,\ \text{by}\,\ \left( f,\Phi \right)^{\leftarrow}(b) =  \Phi^{op} \circ b\circ f,
$
in other words, that diagram
\[
\begin{diagram}
\node{X}\arrow{e,t}{f}\arrow{s,l}{\left( f,\Phi \right)^{\leftarrow}(b)}\node{Y}\arrow{s,r}{b}\\
\node{L}\node{M}\arrow{w,b}{\Phi^{op}}
\end{diagram}
\]
is commutative.
\end{enumerate}
Note that these operators were defined taking into account the Adjoint functor theorem. Consequently, we have that $f^{\rightarrow}$, \, $f_L^{\rightarrow}$, and $\left( f,\Phi \right)^{\rightarrow}$ are left adjoints of $f^{\leftarrow}$,\, $f_L^{\leftarrow}$, and $\left( f,\Phi \right)^{\leftarrow}$, respectively.

\section{Basic properties of variable-basis interior operators}
In this section we consider a subcategory $\mathcal{D}$ of CQML in order to construct fuzzy variable-basis interior operators on the category $SET\times\mathcal{D}$ that has as objects all pairs $(X,L)$, where $X$ is a set and $L$ is an object of $\mathcal{D}$, as morphisms from $(X, L)$ to $(Y,M)$ all pairs of maps $(f, \phi)$ with $f\in SET(X,Y)$ and $ \phi \in CQML(L,M)$, identities given by $id_{\text{\tiny{$(X,L)$}}} =(id_{\text{\tiny{$X$}}} ,id_{\text{\tiny{$L$}}}) $, and composition defined by

\[
(f, \phi ) \circ(g, \psi)=(f\circ g, \phi\circ \psi).
\]

\begin{defi} \label{def-int}
 An interior operator of the category $SET\times\mathcal{D}$ is given by a family

$I=(i_{\text{\tiny{$XL$}}})_{\text{\tiny{$(X,L)$}} \in \left| SET\times\mathcal{D} \right|}$ of maps 
$i_{\text{\tiny{$XL$}}}:L^{\text{\tiny{$X$}}}\longrightarrow L^{\text{\tiny{$X$}}}$ that satisfies the requeriment:

\begin{enumerate}
\item[($I_1$)] (Contraction) $i_{\text{\tiny{$X$}}}(u)\ls u\quad \text{for all}\quad u \in L^X$;
\item[($I_2$)] (Monotonicity)  if $u\ls v$ in $L^X$, then $\ix(u)\ls \ix(v)$;
\item[($I_3$)] (Upper bound) $\ix(1_{\text{\tiny{$X$}}})=1_{\text{\tiny{$X$}}}$.
\end{enumerate}
\end{defi}
                         
 \begin{defi} \label{int}
 A fuzzy variable-basis $I$-space is a triple $(X,L, i_{\text{\tiny{$XL$}}})$, where $(X,L)$ is an
object of $SET\times\mathcal{D}$ and $i_{\text{\tiny{$XL$}}}$ is an interior map on $(X,L)$.
\end{defi}

\begin{defi}
 A morphism $(f, \phi):(X,L) \longrightarrow (Y,M)$ in $SET\times\mathcal{D}$ is said to be fuzzy $I$-continuos if

\begin{equation}\label{ic1}
(f, \phi)^{\leftarrow}\big(i_{\text{\tiny{$YM$}}}(v)\big) \ls i_{\text{\tiny{$XL$}}}\Big((f,\phi)^\leftarrow (v)\Big)\; \text{for all} \; v \in M^{\text{\tiny{$Y$}}} . 
\end{equation}
\end{defi}

\begin{prop}
Consider two fuzzy $I$-continuous morphisms \linebreak $(f, \phi):(X,L) \longrightarrow (Y,M)$ and $(g,\psi):(Y,M)\longrightarrow (Z,N)$, then the morphism $(g, \psi)\circ (f,\phi)$ is fuzzy $I$-continuous.
\end{prop}

\begin{proof}
Since $(g, \psi):(Y,M) \longrightarrow (Z,N)$ is $I$-continuous we have

\[
(g,\psi)^\leftarrow(i_{\text{\tiny{$ZN$}}}(w))\ls i_{\text{\tiny{$YM$}}}\big((g,\psi)^\leftarrow (w)\big) \quad \text{for all} \quad w \in N^{\text{\tiny{$Z$}}}
\] 
it follows that
 \[
(f,\phi)^\leftarrow \Big( (g, \psi)^\leftarrow  \big(i_{\text{\tiny{$ZN$}}} (w)\big)\Big)\ls(f,\phi)^\leftarrow\Big(i_{\text{\tiny{$YM$}}}\big((g,\psi)^\leftarrow (w)\big)\Big)
 \]

now, by the fuzzy $I$-continuity of $(f,\phi)$,
\[
(f,\phi)^\leftarrow(i_{\text{\tiny{$YM$}}}(v))\ls i_{\text{\tiny{$XL$}}}\big((f,\phi)^\leftarrow (v)\big) \quad \text{for all} \quad v \in M^{\text{\tiny{$Y$}}},
\]
in particular for $v=(g,\psi)^\leftarrow (w)$,
\[
(f,\phi)^\leftarrow\Big(i_{\text{\tiny{$YM$}}}\big((g,\psi)^\leftarrow (w)\big)\Big)\ls i_{\text{\tiny{$XL$}}}\Big((f,\phi)^\leftarrow \big((g,\psi)^\leftarrow (w)\big)\Big), 
\]

therefore 

\[
\Big((g, \psi)\circ (f,\phi)\Big)^\leftarrow\big(i_{\text{\tiny{$ZN$}}}(w)\big)\ls i_{\text{\tiny{$XL$}}}\Big(\big((g, \psi)\circ (f,\phi)\big)^\leftarrow (w)\Big).
\]
\end{proof}
   As a consequence we obtain
    
 \begin{defi} The category VBIO-SET that has as objects all triples $(X,L,i_{\text{\tiny{$XL$}}})$ where $(X,L)$ is an object of $SET\times\mathcal{D}$ and $i_{\text{\tiny{$XL$}}}:L^{\text{\tiny{$X$}}}\longrightarrow L^{\text{\tiny{$X$}}}$ is a fuzzy interior map, as morphisms from $(X,L,i_{\text{\tiny{$XL$}}})$ to $(Y,M,i_{\text{\tiny{$YM$}}})$ all pairs of fuzzy $I$-continuous functions $(f,\phi):(X,L,i_{\text{\tiny{$XL$}}})\longrightarrow (Y,M,i_{\text{\tiny{$YM$}}})$, identities and composition as in $SET\times D$
 \end{defi} 
 
\subsection{ The lattice structure of all interior operators.}

 We consider the collection
\[I\big( SET, L\big)\, \text{ for all}\, L \in \mathcal{D}
\]
of all interiors operators on $ SET\times \{L\} $. It is ordered by
$$I \ls J \leftrightarrow i_{\text{\tiny{$X$}}} (u) \ls j_{\text{\tiny{$X$}}}(u),\, \text{for all set $X$, and for all}\, u\in L^X.$$
This way $I\big(SET,L\big)$ inherents a lattice structure from $L$:

\begin{prop}
Every family $\big(I_{\text{\tiny{$\lambda$}}}\big)_{\text{\tiny{$\lambda \in \Lambda$}}}$ in $I\big(SET,L\big)$ has a join $\bigvee \limits_{\lambda \in \Lambda} I_{\text{\tiny{$\lambda$}}}$ and a meet $\bigwedge\limits_{\lambda \in \Lambda} I_{\text{\tiny{$\lambda$}}}$ in $I\big(SET,L\big)$.
The discrete interior operator
\[ 
I_{\text{\tiny{$D$}}}= \big(i_{\text{\tiny{$D_X$}}}\big)_{X \in |SET|}\,\ \text{with}\,\ i_{\text{\tiny{$D_X$}}(u)}=u \,\ \text{for all}\,\ u\in L^X
\]
is the largest element in $I\big(SET,L\big)$, and the trivial interior operator
\[ I_{\text{\tiny{$T$}}}= \big(i_{\text{\tiny{$T_X$}}}\big)_{X \in |SET|}\qquad \text{with}\qquad \big(i_{\text{\tiny{$T_X$}}}\big)(u)=\begin{cases}&1_{\text{\tiny{$X$}}} \; \text{for all} \,\,\ u\neq 0 \\ & 0_{\text{\tiny{$X$}}}\; \text{if} \,\,\ u=0_{\text{\tiny{$X$}}}
\end{cases}
\]
is the laeast one.

\end{prop}

\begin{proof}
For $\Lambda \neq \emptyset $, let $\tilde{I}=\bigvee\limits_{\lambda \in \Lambda} I_{\text{\tiny{$\lambda$}}}$, then  
$$\tilde{i}_{\text{\tiny{$X$}}}=\bigvee \limits_{\lambda \in \Lambda} i_{\text{\tiny{$\lambda_{X}$}}},$$ where $X$ is an arbitrary set, satisfies 
\begin{itemize}
\item $ \tilde{i}_{\text{\tiny{$X$}}}(u)\ls u$, because $i_{\text{\tiny{$\lambda_{X}$}}}(u)\ls u$ for all $u \in L^X$ and for all $\lambda \in \Lambda.$ 
\item If $u_1 \ls u_2$ in $L^X$ then $ i_{\text{\tiny{$\lambda_{X}$}}}(u_1)\ls i_{\text{\tiny{$\lambda_{X}$}}}(u_2)$ for all $\lambda \in \Lambda$, therefore $\tilde{i}_{\text{\tiny{$X$}}}(u_1)\ls \tilde{i}_{\text{\tiny{$X$}}}(u_2).$
\item  Since $i_{\text{\tiny{$\lambda_{X}$}}}(0_X)=0_X$ for all $\lambda \in \Lambda$, we have that $\tilde{i}_{\text{\tiny{$X$}}}(0_X)=0_X.$

\end{itemize}

Similary $\bigvee\limits_{\lambda \in \Lambda} I_{\text{\tiny{$\lambda_{X}$}}}$, $I_{\text{\tiny{$T_X$}}}$ and $I_{\text{\tiny{$D_X$}}}$ are interior operators.

\end{proof}

Consequently,
\begin{coro}\label{cl}
For every set $X$
$$I(X)=\{i_{\text{\tiny{$X$}}} \mid  i_{\text{\tiny{$X$}}} \, \text{is an interior map on X}\}$$ is a complete lattice.
\end{coro}

\subsection{Initial variable-basis interior operator}   
 Let $(Y,M,i_{\text{\tiny{$YM$}}})$ be an object of the category   VBIO-SET and let $(X,L)$ be an object of the category $SET\times\mathcal{D}$.\\
  For each morphism $(f,\phi):(X,L)\longrightarrow(Y,M)$ in $SET\times\mathcal{D}$ we define on $(X,L)$ the map
 $\hat{i}_{\text{\tiny{$XL$}}}:L^{\text{\tiny{$X$}}}\longrightarrow L^{\text{\tiny{$X$}}}$ by 
 \begin{equation}\label{ec*}
 \hat{i}_{\text{\tiny{$XL$}}}=(f,\phi)^\leftarrow \circ i_{\text{\tiny{$YM$}}}\circ (f,\phi)_{\ast},
 \end{equation}
 where $(f,\phi)_{\ast}$ is the right adjoint of $(f,\phi)^\leftarrow$. In other words, the following diagram is conmutative
 \[
\begin{diagram}
\node{L^{\text{\tiny{$X$}}}}\arrow{e,t}{(f,\phi)_{\ast}}\arrow{s,..}{\hat{i}_{\text{\tiny{$XL$}}}}\node{M^{\text{\tiny{$Y$}}}}\arrow{s,r}{i_{\text{\tiny{$YM$}}}}\\
\node{L^{\text{\tiny{$X$}}}}\node{M^{\text{\tiny{$Y$}}}}\arrow{w,b}{(f,\phi)^\leftarrow}
\end{diagram}
\]

\begin{prop}\label{imap}
$\hat{i}_{LX}$ an interior map.
\end{prop}

\begin{proof}

\begin{itemize}

\item[(1)](Contraction) \begin{align*}
\hat{i}_{\text{\tiny{$LX$}}}(u)&= \Big((f,\phi)^\leftarrow \circ i_{\text{\tiny{$MY$}}} \circ (f,\phi)_{\ast}\Big)(u)\\ &\leq\Big((f,\phi)^\leftarrow  \circ (f,\phi)_{\ast}\Big)(u) \leq u;
\end{align*}
 
 \item[(2)] (Monotonicity) If $u_1 \leq u_2$ \quad then \quad $(f,\phi)_{\ast}(u_1)\leq (f,\phi)_{\ast}(u_2)$, \\
 therefore 
 $$i_{\text{\tiny{$MY$}}}\big((f,\phi)_{\ast}(u_1)\big)\leq i_{\text{\tiny{$MY$}}} \big((f,\phi)_{\ast}(u_2)\big),$$
 consequently
 $$(f,\phi)^\leftarrow \Big(i_{\text{\tiny{$MY$}}} \big((f,\phi)_{\ast}(u_1)\big)\Big)\leq (f,\phi)^\leftarrow \Big(i_{\text{\tiny{$MY$}}}\big((f,\phi)_{\ast}(u_2)\big)\Big)$$
 that is,\quad
 $\hat{i}_{LX}(u_1)\leq \hat{i}_{LX}(u_2)$;\\
 
 \item[(3)](Upper bound)
 \begin{align*}
 \hat{i}_{\text{\tiny{$LX$}}}(1_X)&=(f,\phi)^\leftarrow \circ i_{\text{\tiny{$MY$}}} \circ (f,\phi)_{\ast}(1_X)=(f,\phi)^\leftarrow \circ i_{\text{\tiny{$MY$}}}(1_Y)\\ &=(f,\phi)^\leftarrow(1_Y)= 1_X.
 \end{align*}
\end{itemize} 
\end{proof}
 \subsubsection{Structural source}
  \begin{prop}\label{ss}
Let $(X,L)$ be an object of $SET\times\mathcal{D}$, let $(Y_\lambda,M_\lambda, i_{\text{\tiny{$Y_\lambda M_\lambda$}}})$ be a family of fuzzy variable $I$-spaces, where $\lambda\in \Lambda$ for some indexed set $\Lambda$, and let $(f_\lambda, \phi_\lambda):(X,L)\rightarrow (Y_\lambda, M_\lambda)$ be a family of morphisms in $SET\times\mathcal{D}.$ Then the structured source $\left[(X,L),\Big((f_\lambda, \phi_\lambda),(Y_\lambda, M_\lambda) \Big)\right]_{\lambda\in\Lambda}$  w.r.t the forgetful functor $U$ from VBIO-SET to $SET\times\mathcal{D}$ has a unique initial lift $\big( (X,L,\tilde{i}_{\text{\tiny{$XL$}}})\rightarrow (Y_\lambda,M_\lambda, i_{\text{\tiny{$Y_\lambda M_\lambda$}}})\big)$, where $\tilde{i}_{\text{\tiny{$XL$}}}$ is the meet  $\bigwedge \limits_{\lambda \in \Lambda} \, ^{\text{\tiny{$\lambda$}}}\hat{i}_{\text{\tiny{$XL$}}}$ of all initial interior maps $^{\text{\tiny{$\lambda$}}}\hat{i}_{\text{\tiny{$XL$}}}$ w.r.t. $(f_\lambda, \phi_\lambda)$,  where $\lambda\in \Lambda$.
\end{prop}
  
  \begin{proof}
We must show that for every object $(Z,N, i_{\text{\tiny{$ZN$}}})$  of VBIO-SET,  each morphism $(g,\psi):(Z,N)\rightarrow (X,L)$ is I-continuous iff each $(f_\lambda, \phi_\lambda)\circ (g,\psi)$ is I-continuous, for all $\lambda\in \Lambda$ and for all $u \in L^{X}$ \\
In fact,
 \begin{align*}
 (g, \psi)^\leftarrow \big(\tilde{i}_{\text{\tiny{$XL$}}} (u)\big)&= (g, \psi)^\leftarrow \big(\bigwedge_{\lambda\in\Lambda}\,  ^{\text{\tiny{$\lambda$}}}\hat{i}_{\text{\tiny{XL}}} (u)\big)\\
 &= \bigwedge_{\lambda\in\Lambda} (g,\psi)^\leftarrow \big( ^{\text{\tiny{$\lambda$}}}\hat{i}_{\text{\tiny{$XL$}}} (u)\big)\\
 &= \bigwedge_{\lambda\in\Lambda} (g, \psi)^\leftarrow \big( (f_{\lambda},\phi_{\lambda})^\leftarrow \circ i_{\text{\tiny{$Y_{\lambda}M_{\lambda}$}}}\circ (f_{\lambda},\phi_{\lambda})_{\ast}(u)\big)\\
 &= \bigwedge_{\lambda\in\Lambda} \left[(g,\psi)^\leftarrow\circ(f_{\lambda},\phi_{\lambda})^\leftarrow \right]\circ i_{\text{\tiny{$Y_{\lambda}M_{\lambda}$}}}\circ (f_{\lambda},\phi_{\lambda})_{\ast}(u),
\end{align*}
 but as the composition  $(f_{\lambda},\phi_{\lambda})\circ (g, \psi)\, \text{for all} \, \lambda\in\Lambda $ is a continous then
  
  \begin{align*}
 (g, \psi)^\leftarrow \big(\tilde{i}_{\text{\tiny{$XL$}}} (u)\big)&
 \leq \bigwedge_{\lambda\in\Lambda}i_{\text{\tiny{$ZN$}}}\left[(g, \psi)^\leftarrow\circ(f_{\lambda},\phi_{\lambda})^\leftarrow \right] \circ (f_{\lambda},\phi_{\lambda})_{\ast}(u)\\
&= \bigwedge_{\lambda\in\Lambda}i_{\text{\tiny{$ZN$}}}(g, \psi)^\leftarrow\circ\left[(f_{\lambda},\phi_{\lambda})^\leftarrow  \circ (f_{\lambda},\phi_{\lambda})_{\ast}\right](u)\\
&\leq\bigwedge_{\lambda\in\Lambda} i_{\text{\tiny{$ZN$}}}(g, \psi)(u)\\
&= i_{\text{\tiny{$ZN$}}}(g,\psi)^\leftarrow(u)
\end{align*}
then 
$$(g, \psi)^\leftarrow \big(\tilde{i}_{\text{\tiny{$XL$}}} (u)\big)\leq i_{\text{\tiny{$ZN$}}}(g,\psi)^\leftarrow (u) .$$

\end{proof}

  As a consequence of corollary (\ref{cl}), proposition (\ref{imap}) and proposition (\ref{ss}), we obtain
          
   \begin{theorem} The concrete category $(\text{VBIO-SET}, U)$ over $SET\times\mathcal{D}$ is a topological category.
          \end{theorem}

\section{Some additional properties of interior operators}
 \begin{defi}
The interior operator $I=(i_{\text{\tiny{$XL$}}})_{\text{\tiny{$(X,L)\in |SET\times \mathcal{D}|$}}}$ of definition (\ref{int}) is called idempotent if the condition
\[
i_{\text{\tiny{$XL$}}}\big(i_{\text{\tiny{$XL$}}}(u)\big)= i_{\text{\tiny{$XL$}}}(u)\quad \text{for all}\quad u\in L^X
\]
holds for every pair $(X,L)\in |SET\times \mathcal{D}|$.
\end{defi}

\begin{prop}
Let $I=(i_{\text{\tiny{$XL$}}})_{\text{\tiny{$(X,L)\in |SET\times \mathcal{D}|$}}}$ be an  idempotent interior operator. Then 
the initial interior operator $\hat{I}=(\hat{i}_{\text{\tiny{$XL$}}})_{\text{\tiny{$(X,L)\in |SET\times \mathcal{D}|$}}}$ defined by

\[
 \hat{i}_{\text{\tiny{$XL$}}}=(f,\phi)^\leftarrow \circ i_{\text{\tiny{$YM$}}}\circ (f,\phi)_{\ast},
\]
 for each morphism $(f,\phi):(X,L)\longrightarrow(Y,M)$ in $SET\times\mathcal{D}$ is also idempotent.
\end{prop}
\begin{proof}

Suppose that $I=(i_{\text{\tiny{$XL$}}})_{\text{\tiny{$(X,L)\in |SET\times \mathcal{D}|$}}}$ is an  idempotent interior operator and let $(f,\phi):(X,L)\longrightarrow(Y,M)$ be a morphism. Then
\begin{align*}
 \hat{i}_{\text{\tiny{$XL$}}}\circ \hat{i}_{\text{\tiny{$XL$}}} &=\Big((f,\phi)^\leftarrow \circ i_{\text{\tiny{$YM$}}}\circ (f,\phi)_{\ast}\Big)\circ \Big((f,\phi)^\leftarrow \circ i_{\text{\tiny{$YM$}}}\circ (f,\phi)_{\ast}\Big)\\
& \geq (f,\phi)^\leftarrow \circ \big(i_{\text{\tiny{$YM$}}}\circ i_{\text{\tiny{$YM$}}}\big)\circ (f,\phi)_{\ast}\\
&= (f,\phi)^\leftarrow \circ  i_{\text{\tiny{$YM$}}}\circ (f,\phi)_{\ast}\\
&= \hat{i}_{\text{\tiny{$XL$}}}.
 \end{align*}
 On the hand, the monotonicity  condition of interior operators implies that 
 $$\hat{i}_{\text{\tiny{$XL$}}}\leq \hat{i}_{\text{\tiny{$XL$}}}\circ \hat{i}_{\text{\tiny{$XL$}}}.$$
 \end{proof}
 
 \begin{defi}
 The interior operator $I=(i_{\text{\tiny{$XL$}}})_{\text{\tiny{$(X,L)$}} \in \left| SET\times\mathcal{D} \right|}$ of definition (\ref{def-int}) is called 
 \begin{enumerate}
 \item productive if the condition
 $$ i_{\text{\tiny{$XL$}}} (u \land v) = i_{\text{\tiny{$XL$}}} (u) \land i_{\text{\tiny{$XL$}}} (v)\quad \text{for all $u,v \in L^{X}$}$$
 holds for every set $X$.
 \item Fully productive if the condition 
 $$i_{\text{\tiny{$XL$}}} \big(\bigwedge_{\lambda\in\Lambda} u_\lambda\big)= \bigwedge_{\lambda\in\Lambda}i_{\text{\tiny{$XL$}}}(u_\lambda)\quad \text{for all} \quad \{u_\lambda \mid \lambda \in \Lambda\}\subseteq L^{X}\, \text{holds for every set $X$}. $$
 \end {enumerate}
  
 \end{defi}
 
 \begin{prop}
 Let $I=(i_{\text{\tiny{$XL$}}})_{\text{\tiny{$(X,L)$}} \in \left| SET\times\mathcal{D} \right|}$ be a fully productive interior operator. Then the initial interior operator $\hat{I}=(\hat{i}_{\text{\tiny{$XL$}}})_{\text{\tiny{$(X,L)$}} \in \left| SET\times\mathcal{D} \right|}$ defined by \\
 \centerline{
 $\hat{i}_{\text{\tiny{$XL$}}}=(f,\phi)^\leftarrow \circ i_{\text{\tiny{$YM$}}}\circ (f,\phi)_{\ast}$ for each morphism $(f,\phi):(X,L) \rightarrow (Y,M)$ 
 }\\ is also fully productive
 \end{prop}
 \begin{proof}
 Supponse that $I=(i_{\text{\tiny{$XL$}}})_{\text{\tiny{$(X,L)$}} \in \left| SET\times\mathcal{D} \right|}$ be a fully productive interior operator and let $f:X \rightarrow Y$ be a funtion. Since $(f, \phi)_{\ast}$ is a right adjont, it preserves all existing meets, and since $(f, \phi)^\leftarrow$ is both left and right adjont, it preserves all existing joins and meets, so for all $\{u_\lambda \mid \lambda \in \Lambda\}\subseteq L^{X}$ 
 
 \begin{align*}
 \hat{i}_{\text{\tiny{$XL$}}}  \big(\bigwedge_{\lambda\in\Lambda} u_\lambda\big)&= (f, \phi)^\leftarrow \Big( i_{\text{\tiny{$YM$}}}\big((f, \phi)_{\ast}(\bigwedge_{\lambda\in\Lambda} u_\lambda)\big)\Big)\\
 &=(f, \phi)^\leftarrow \Big( i_{\text{\tiny{$YM$}}}\big(\bigwedge_{\lambda\in\Lambda} (f, \phi)_{\ast}(u_\lambda)\big)\Big)\\
 &=(f, \phi)^\leftarrow \Big(\bigwedge_{\lambda\in\Lambda} i_{\text{\tiny{$YM$}}}\big( (f, \phi)_{\ast}(u_\lambda)\big)\Big)\\
 &=\bigwedge_{\lambda\in\Lambda}(f, \phi)^\leftarrow \Big( i_{\text{\tiny{$YM$}}}\big( (f, \phi)_{\ast}(u_\lambda)\big)\Big)\\
 &=\bigwedge_{\lambda\in\Lambda}\hat{i}_{\text{\tiny{$XL$}}}
 \end{align*}
 \end{proof}
 \pagebreak
\subsection{Open  fuzzy sets}
          \begin{defi}
          An $L$-fuzzy subset $u$ of $X$ is called $I$-open in $(X,L)$ if it is equal to it interior, i.e, $i_{\text{\tiny{$XL$}}}(u)=u.$ The fuzzy $I$-continuity condition (\ref{ic1}) implies that $I$-openness is preserved by inverse images.
          \end{defi}
          \begin{prop} Let $(f,\phi):(X,L,c_{\text{\tiny{$XL$}}})\longrightarrow(Y,M,c_{\text{\tiny{$YM$}}})$ be a morphism in VBIO-SET. If $v \in M^{\text{\tiny{$Y$}}}$ is $I$-open then $(f,\phi)^\leftarrow(v)$ is $i$-open in $(X,L).$
          \end{prop}
\begin{proof}
   If $v=i_{\text{\tiny{$YM$}}}(v),$ for $v\in M^{\text{\tiny{$Y$}}},$ then $(f,\phi)^\leftarrow(v)=(f,\phi)^\leftarrow\big(i_{\text{\tiny{$YM$}}}(v)\big)\leqslant i_{\text{\tiny{$XL$}}}\big((f,\phi)^\leftarrow(v)\big),$ so $i_{\text{\tiny{$XL$}}}\big((f,\phi)^\leftarrow(v)\big)=(f,\phi)^\leftarrow(v)$.
\end{proof}

    \subsection{$I$-open morphisms}
    \begin{defi}
    A morphism $(f,\phi):(X,L,i_{\text{\tiny{$XL$}}})\longrightarrow(Y,M,i_{\text{\tiny{$YM$}}})$ between  variable-basis fuzzy $I$-spaces is $I$-open if

    \begin{equation}\label{ic2}
i_{\text{\tiny{$XL$}}}\Big((f,\phi)^\leftarrow (v)\Big)\le  (f, \phi)^{\leftarrow}\big(i_{\text{\tiny{$YM$}}}(v)\big)\; \text{for all} \; v \in L^{\text{\tiny{$Y$}}} . 
    \end{equation}
    \end{defi} 
    \begin{prop}
     Let $(f,\phi):(X,L,i_{\text{\tiny{$XL$}}})\longrightarrow(Y,M,i_{\text{\tiny{$YM$}}})$ and\linebreak $(g,\psi):(Y,M,i_{\text{\tiny{$YM$}}})\longrightarrow(Z,N,i_{\text{\tiny{$ZN$}}})$ be two  $I$-open morphisms, then the morphism $(f,\phi)\circ (g,\psi)$ is $I$-open.
        \end{prop}
        
        \begin{proof} Since $(g,\psi):(Y,M,i_{\text{\tiny{$YM$}}})\longrightarrow(Z, N,i_{\text{\tiny{$YM$}}})$ is  $I$-open, we have $$i_{\text{\tiny{$YM$}}}\big((g,\psi)^\leftarrow(w)\big)\ls(g,\psi)^\leftarrow\big( i_{\text{\tiny{$ZN$}}}(w)\big)\quad \text{for all}\quad w \in N^{\text{\tiny{$Z$}}},$$
         it follows that 
         $$(f,\phi)^\leftarrow\Big( i_{\text{\tiny{$YM$}}}\big((g,\psi)^\leftarrow(w)\big)\Big)\ls (f,\phi)^\leftarrow \Big((g,\psi)^\leftarrow\big(i_{\text{\tiny{$ZN$}}}(u)\big)\Big)$$
        
  now, by the  $I$-openness of $(f,\phi),$
  $$i_{\text{\tiny{$XL$}}}\big((f,\phi)^\leftarrow(v)\big)\ls(f,\phi)^\leftarrow\big(i_{\text{\tiny{$YM$}}}(v)\big)\quad \text{for all}\quad v \in M^{\text{\tiny{$Y$}}},$$ in particular for $v=(g,\psi)^\leftarrow(u),$ 
  $$i_{\text{\tiny{$XL$}}}\big((f,\phi)^\leftarrow((g,\psi)^\leftarrow(w))\big)\ls(f,\phi)^\leftarrow\big(i_{\text{\tiny{$YM$}}}((g,\psi)^\leftarrow(w))\big)$$
  therefore
  $$i_{\text{\tiny{$XL$}}}\Big(\big((g,\psi)\circ (f,\phi)\big)^\leftarrow(w)\Big)\ls (\big((g,\psi)\circ (f,\phi)\big)^\leftarrow \big(i_{\text{\tiny{$ZN$}}}(w)\big)$$.
       \end{proof}
If we replace in the category VBIO-SET fuzzy $I$-continuous morphisms by  $I$-open morphisms, we obtain another topological category. 
The morphisms $(f,\phi):(X,L,i_{\text{\tiny{$XL$}}})\longrightarrow(Y,M,i_{\text{\tiny{$YM$}}})$ between  variable-basis fuzzy  $I$-spaces which are bijective,  fuzzy $I$-continuous and $I$-open, forms a group. We can say that a way of seeing  variable-basis fuzzy topology is studying invariants of the action of these groups aver the category $SET\times\mathcal{D}$.

  \section {Some examples of fuzzy interior operators}
\begin{ex}
Let  $L= [0, 1]$ be the unit interval considered as an ordered subset of the real numbers $\mathbb R$, as well as a  complete (not complemented) lattice. 
\begin{enumerate}
\item[(i)] For each  topological space $X$, we define $i_{\text{\tiny{$X$}}}: I^X\rightarrow L^X$ by
\[
i_{\text{\tiny{$X$}}}(u)=\bigvee\{ v\in L^X\mid v \,\ \text{is lower semi-continuous and}\,\ v\leqslant u\}.
\]
Clearly, the family $I=(i_{\text{\tiny{$X$}}})_{\text{\tiny{$X\in |TOP|$}}}$  is  a  fuzzy interior operator of the category $TOP$. Since the fixed points of the restriction of $i_{\text{\tiny{$X$}}}$ to $2^X$ produces the open sets of $X$, this operator is an extension of the usual interior in $TOP$.
\item[(ii)] For each  compact topological space $Y$, we define $j_{\text{\tiny{$Y$}}}: L^Y\rightarrow L^Y$ by
$j_{\text{\tiny{$Y$}}}(v)=m_v ,$ 
where $m_v$  is the constant function on $L^Y$ whose value is $ m_v=\min\{v(y)\mid y\in Y\}$.

Undoubtedly, the family $D=(d_{\text{\tiny{$Y$}}})_{\text{\tiny{$Y\in |COMP|$}}}$  is  a  fuzzy interior operator of the category $COMP$ (of compact topological spaces).
\item[(iii)] Every map $f:X\rightarrow Y$ from a topological space $X$ to a compact space $Y$ is fuzzy IJ-continuous since each constant map is lower semi-coninuous.
\item[(iv)] On the other hand, the only fuzzy JI-continuous maps between compact spaces an topological spaces are the constant.
\end{enumerate}
\end{ex}
\begin{ex}
Let $X = \{x\}$ be a single point set and $L = [0, 1]$ be the usual unit interval. The  maps $i_{\text{\tiny{$n$}}}:L^X\rightarrow L^X$ defined by $i_{\text{\tiny{$n$}}}(t)=t^{n}$, for $n=1,2,\cdots$ are interior maps, from which just $i_{\text{\tiny{$1$}}}$ and $i_{\text{\tiny{$\infty$}}}=\lim\limits_{n\to \infty}t^{n}$ are idempotent.
\end{ex}
\begin{ex}
For a $GL-$monoid $\L$ and for an $\L$-topology $\tau\subseteq \L^X$, we define
\[
i_{\text{\tiny{$X$}}}(u)=\bigvee\{v\in \tau\mid  v\ls u\}.
\]
These  maps produce an interior operator of the category L-TOP  whose associated closure operator is
\[
c_{\text{\tiny{$X$}}}(u)=\bigwedge_{v\in \tau}\{v\impl 0_X\mid  u\ls v\}.
\]
\end{ex}



\begin{thebibliography}{10}
\bibitem{AHS}  Jiri Adamek, Horst Herrlich, George Strecker, 1990, ``Abstract and Concrete Categories'', John Wiley \& Sons, New York.
\bibitem{GB} G. Birkhoff, 1940, ``Lattice Theory'', American Mathematical Society, Providence. 
\bibitem{NB} N. Bourbaki, 1966, ``General Topology'', Addison-Wesley Publishing, Massachusetts.
\bibitem{DDU} M. Diker, S. Dost and A. U$\check{g}$ur, 2009, ``Interior and closure operators on texture spaces—I: Basic concepts and $\check{C}$ech closure operators'', Fuzzy Sets and Systems 161 pp. 935–953.
\bibitem{HS} U. H\"ohle, A. \v{S}ostak, 1999, ``Fixed-basis fuzzy topologies'', In: Mathematics of Fuzzy Sets: Logic, Topology and Measure Theory, Kluwer Academic Publisher, Boston.
\bibitem{MM}  S. MacLane, I. Moerdijk, 1992, ``Sheaves in Geometry and Logic,{ \scriptsize A first introduction to Topos theory}'', Springer-Verlag, New York / Heidelberg / Berlin.
\bibitem{SER} S. E. Rodabaugh, 1999 ``Powerset operator foundations for poslat fuzzy set theories and topologies'', In: Mathematics of Fuzzy Sets: Logic, Topology and Measure Theory, Kluwer Academic Publisher, Boston.
\bibitem{SER1} S. E. Rodabaugh, 1999,  ``Categorical foundations of variable-basis fuzzy topology'', In: Mathematics of Fuzzy Sets: Logic, Topology and Measure Theory, Kluwer Academic Publisher, Boston.
\bibitem{FS}  F. G. Shi, 2009,  ``L-fuzzy interiors and L-fuzzy closures'', Fuzzy Sets and Systems 160, pp. 1218-1232. 
\bibitem{SL} W. Shi, K. Liu, 2007, ``A fuzzy topology for computing the interior, boundary, and exterior of spatial objects quantitatively in GIS'', Computers \& Geosciences 33, pp. 898–915.


\end{thebibliography}
 \end{document}